\documentclass{article}
\usepackage{amsmath}
\usepackage{amsthm}
\usepackage{amssymb}
\usepackage{graphicx}
\usepackage{graphics}
\usepackage{listings}
\usepackage{multirow}
\usepackage{framed}
\usepackage{hyperref}

\newcommand{\set}[1]{\left\{#1\right\}}

\usepackage{color}

\usepackage[algo2e,ruled,vlined,]{algorithm2e}
\usepackage{algorithmic}

\definecolor{red}{rgb}{1,0,0}

\definecolor{blue}{rgb}{0,0,1}

\usepackage{mathrsfs}

\newcommand{\ms}{\medskip}
\newcommand{\bpf}{\begin{proof}}
\newcommand{\epf}{\end{proof}\ms}

\newtheorem{theorem}{Theorem}
\newtheorem{corollary}[theorem]{Corollary}
\newtheorem{lemma}[theorem]{Lemma}
\newtheorem{proposition}[theorem]{Proposition}
\newtheorem{observation}[theorem]{Observation}

\theoremstyle{definition}
\newtheorem{definition}{Definition}
\newtheorem{question}{Question}

\begin{document}

\title{Power domination polynomials of graphs}

\author{Boris Brimkov\thanks{Department of Computational and Applied Mathematics, Rice University, Houston, TX, 77005, USA (boris.brimkov@rice.edu, rsp7@rice.edu, vs30@rice.edu, alexander.b.teich@rice.edu)} , 
Rutvik Patel$^*$,
Varun Suriyanarayana$^*$,
Alexander Teich$^*$}
\date{}

\maketitle

\begin{abstract}

A power dominating set of a graph is a set of vertices that observes every vertex in the graph by combining classical domination with an iterative propagation process arising from electrical circuit theory. In this paper, we study the power domination polynomial of a graph $G$ of order $n$, defined as $\mathcal{P}(G;x)=\sum_{i=1}^n p(G;i) x^i$, where $p(G;i)$ is the number of power dominating sets of $G$ of size $i$. We relate the power domination polynomial to other graph polynomials, present structural and extremal results about its roots and coefficients, and identify some graph parameters it contains. We also derive decomposition formulas for the power domination polynomial, compute it explicitly for several families of graphs, and explore graphs which can be uniquely identified by their power domination polynomials.

\smallskip

\end{abstract}

\section{Introduction}
Let $G=(V,E)$ be a graph, let $S\subset V$ be a set of initially colored vertices, and consider the following color change rules:
\begin{enumerate}
\item[1)] Every neighbor of an initially colored vertex becomes colored.
\item[2)] Whenever there is a colored vertex with exactly one uncolored neighbor, that neighbor becomes colored.
\end{enumerate}
$S$ is a {\em power dominating set} of $G$ if all vertices in $G$ become colored after applying rule 1) once, and rule 2) as many times as possible (i.e. until no more vertices can change color). The application of rule 1) is called the \emph{domination step}, and each application of rule 2) is called a \emph{forcing step}. The {\em power domination number} of $G$, denoted $\gamma_P(G)$, is the cardinality of a minimum power dominating set. $S$ is a {\em zero forcing set} of $G$ if all vertices in $G$ become colored after applying rule 2) as many times as possible (and not applying rule 1) at all). The {\em zero forcing number} of $G$, denoted $Z(G)$, is the cardinality of a minimum zero forcing set. $S$ is a {\em dominating set} of $G$ if all vertices in $G$ become colored after applying rule 1) once; the cardinality of a minimum dominating set is called the \emph{domination number}, denoted $\gamma(G)$. Thus, the power domination process can be regarded as a combination of classical graph domination and zero forcing.

Power domination arises from a graph theoretic model of the Phase Measurement Unit (PMU) placement problem from electrical engineering: in order to monitor their networks, electrical power companies place PMUs at select locations in the power network; the physical laws by which PMUs observe the network give rise to the color change rules described above (cf. \cite{BH05,powerdom3}). This PMU placement problem has been explored extensively in the electrical engineering literature; see \cite{EEprobabilistic,Baldwin93,Brunei93,EEinformationTheoretic,EEtaxonomy,Mili91,EEtabuSearch,EEmultiStage}, and the bibliographies therein for various placement strategies and computational results. The process of zero forcing was introduced independently in combinatorial matrix theory \cite{AIM-Workshop} and in quantum control theory \cite{quantum1}. See, e.g., \cite{benson,brimkov_caleb,connected_zf,caro,dorbec2,dorfling,kneis,row,xu} for various structural and computational results about power domination, zero forcing, and related variants.

In this paper, we study the counting problem associated with power domination, i.e., characterizing and counting the distinct power dominating sets of a graph. The set of minimum power dominating sets of a graph has been used previously in the context of \emph{power propagation time} \cite{aazami,dorbec,ferrero,liao}, where the objective is to find the smallest number of timesteps it takes for the graph to be colored by a minimum power dominating set. Sets of vertices related to complements of power dominating sets (of arbitrary size) have also been used in integer programming approaches for computing the power domination number (cf. \cite{bozeman}). The PMU placement literature also considers power dominating sets with various additional properties, such as redundancy, controlled islanding, and connectedness, and optimizes over them in addition to the cardinality of the set (see, e.g., \cite{akhlaghi,connected_pd,mahari,xia}). In order to study the collection of power dominating sets of a graph in a more general framework, we introduce the \emph{power domination polynomial}, which counts the number of distinct power dominating sets of a given size. 
\begin{definition}
Let $G$ be a graph on $n$ vertices and $p(G;i)$ be the number of power dominating sets of $G$ with cardinality $i$. The \emph{power domination polynomial} of $G$ is defined as
\begin{equation*}
\mathcal{P}(G;x)=\sum_{i=1}^n p(G;i) x^i.
\end{equation*}
\end{definition}

\noindent We study the basic algebraic and graph theoretic properties of the power domination polynomial, present structural and extremal results about its roots and coefficients, compute it explicitly for several families of graphs, and identify some graphs which can be recognized by their power domination polynomials. We also relate the power domination polynomial to the \emph{zero forcing polynomial} and the \emph{domination polynomial} of a graph $G$, which respectively count the number of zero forcing sets and dominating sets of $G$. More precisely, the zero forcing polynomial of $G$ is defined as $\mathcal{Z}(G;x)=\sum_{i=1}^n z(G;i) x^i$ and the domination polynomial of $G$ is defined as $\mathcal{D}(G;x)=\sum_{i=1}^n d(G;i) x^i$, where $z(G;i)$ is the number of zero forcing sets of $G$ of size $i$, and $d(G;i)$ is the number of dominating sets of $G$ of size $i$.

Work on the domination polynomial and zero forcing polynomial includes derivations of recurrence relations \cite{dom_poly3}, analysis of the roots  \cite{dom_poly2}, and characterizations for specific graphs \cite{dom_poly4,dom_poly1,zfp,dom_poly5}. Similar results have been obtained for the \emph{connected domination polynomial} \cite{connected_dom_poly}, \emph{independence polynomial} \cite{indep_poly}, \emph{clique polynomial} \cite{clique_poly}, \emph{vertex cover polynomial} \cite{vertex_cover_poly}, and \emph{edge cover polynomial} \cite{edge_cover_poly}, which are defined as generating functions of their eponymous sets. In general, graph polynomials contain important information about the structure and properties of graphs that can be extracted by algebraic methods. In particular, the values of graph polynomials at specific points, as well as their coefficients, roots, and derivatives, often have meaningful interpretations. Such information and other unexpected connections between graph theory and algebra are sometimes discovered long after a graph polynomial is originally introduced (see, e.g., \cite{brylawski,stanley}). For more definitions, results, and applications of graph polynomials, see the survey of Ellis-Monaghan and Merino \cite{tuttemain2} and the bibliography therein.

This paper is organized as follows. In the next section, we recall some graph theoretic notions and notation. In Section \ref{section_struct}, we present a variety of structural and extremal results about the coefficients of the power domination polynomial, relate the power domination polynomial to other graph polynomials, and give several decomposition formulas. In Section \ref{section_characterizations}, we give closed-form expressions for the power domination polynomials of several families of graphs, and explore graphs which can be uniquely identified by their power domination polynomials. In Section \ref{section_roots}, we study the roots of the power domination polynomial. We conclude with some final remarks and open questions in Section \ref{section_conclusion}.

\section{Preliminaries}

A graph $G=(V,E)$ consists of a vertex set $V=V(G)$ and an edge set $E=E(G)$ of two-element subsets of $V$. The \emph{order} of $G$ is denoted by $n(G)=|V|$. We will assume that the order of $G$ is nonzero, and when there is no scope for confusion, dependence on $G$ will be omitted. Two vertices $v,w\in V$ are \emph{adjacent}, or \emph{neighbors}, if $\{v,w\}\in E$; we will sometimes write $vw$ to denote an edge $\{v,w\}$. The \emph{neighborhood} of $v\in V$ is the set of all vertices which are adjacent to $v$, denoted $N(v)$; the \emph{degree} of $v\in V$ is defined as $d(v)=|N(v)|$. The \emph{closed neighborhood} of $v$ is the set $N[v]=N(v)\cup \{v\}$. Given $S \subset V$, $N[S]=\bigcup_{v\in S}N[v]$, and $N(S)=\left(\bigcup_{v\in S}N(v)\right)\backslash S$.

Given $S \subset V$, the \emph{induced subgraph} $G[S]$ is the subgraph of $G$ whose vertex set is $S$ and whose edge set consists of all edges of $G$ which have both endpoints in $S$. An isomorphism between graphs $G_1$ and $G_2$ will be denoted by $G_1\simeq G_2$. An \emph{isolated vertex}, or \emph{isolate}, is a vertex of degree 0. A \emph{dominating vertex} is a vertex which is adjacent to all other vertices. A \emph{cut vertex} is a vertex which, when removed, increases the number of connected components in $G$. A \emph{block} of $G$ is a maximal (with respect to inclusion) subgraph which does not contain cut vertices. The path, cycle, complete graph, and empty graph on $n$ vertices will respectively be denoted $P_n$, $C_n$, $K_n$, $\overline{K_n}$. An \emph{endpoint} of $P_n$ is a degree 1 vertex if $n>1$, and a degree 0 vertex if $n=1$.

\sloppypar Given two graphs $G_1$ and $G_2$, the \emph{disjoint union} $G_1\dot\cup G_2$ is the graph with vertex set $V(G_1)\dot\cup V(G_2)$ and edge set $E(G_1)\dot\cup E(G_2)$. The \emph{join} of $G_1$ and $G_2$, denoted $G_1\lor G_2$, is the graph obtained from $G_1\dot\cup G_2$ by adding an edge between each vertex of $G_1$ and each vertex of $G_2$. The \emph{corona} of two graphs $G_1$ and $G_2$, denoted $G_1 \circ G_2$, is the graph obtained by taking one copy of $G_1$ and $n(G_1)$ copies of $G_2$, and adding an edge between the $i^\text{th}$ vertex of $G_1$ and each vertex of the $i^\text{th}$ copy of $G_2$, $1\leq i\leq n(G_1)$. The \emph{complete bipartite graph} with parts of size $a$ and $b$, denoted $K_{a,b}$, is the graph $\overline{K_a}\lor\overline{K_b}$. The graph $K_{n-1,1}$, $n\geq 3$, will be called a \emph{star} and denoted $S_n$, and the graph $C_{n-1}\lor K_1$, $n\geq 4$, will be called a \emph{wheel} and denoted $W_n$.
For other graph theoretic terminology and definitions, we refer the reader to \cite{bondy}.

Given integers $a$ and $b$ with $0\leq a<b$, we adopt the convention that ${a\choose b}=0$. We adopt also the conventions that $\sum_{i\in \emptyset}s_i=0$, $\prod_{i\in \emptyset} s_i= 1$, and $\bigcup_{i\in \emptyset} S_i= \emptyset$. For any positive integer $n$, $[n]$ denotes the set $\{1,\ldots,n\}$. $\mathbb{N}$, $\mathbb{Z}$, $\mathbb{Q}$, $\mathbb{R}$, and $\mathbb{C}$ respectively denote the set of positive integers, integers, rational numbers, real numbers, and complex numbers.

\section{Structural results}
\label{section_struct}

In this section, we give several structural and extremal results about the coefficients of power domination polynomials. We begin with the following basic observation.

\begin{observation}
\label{obs_subset}
Let $G=(V,E)$ be a graph and $R\subset V$. If $S$ is a power dominating set of $G$, then $S\cup R$ is a power dominating set of $G$. Equivalently, any superset of a power dominating set is power dominating, and any subset of a non-power dominating set is not power dominating.
\end{observation}

\begin{proposition}
\label{prop_all_sets}
Let $G=(V,E)$ be a graph. If $p(G;i)=\binom{n}{i}$ for some $i\in[n]$, then $p(G;j)=\binom{n}{j}$ for $j\in\{i,\ldots,n\}$. 
\end{proposition}

\proof
Note that $p(G;j)\leq \binom{n}{j}$ for all $j\in [n]$. Suppose there exists a $j\in\{i,\ldots,n\}$ such that $p(G;j)<\binom{n}{j}$. Then there exists a set $S\subset V$ of size $j$ which is not power dominating. By Observation \ref{obs_subset}, no subset of $S$ is power dominating; in particular, the subsets of $S$ of size $i$ are not power dominating, which contradicts the assumption that every subset of $V$ of size $i$ is power dominating.
\qed

\vspace{9pt}

\noindent A \emph{matching} of $G=(V,E)$ is a set $M\subset E$ such that no two edges in $M$ have a common endpoint. A matching $M$ \emph{saturates} a vertex $v$, if $v$ is an endpoint of some edge in $M$.

\begin{theorem}[Hall's Theorem \cite{hall}] 
\label{thm_hall}
Let $G$ be a bipartite graph with parts $X$ and $Y$. $G$ has a matching that saturates every vertex in $X$ if and only if for all $S\subset X$, $|S|\leq |N(S)|$.
\end{theorem}

\begin{proposition}
\label{prop_unimodal}
Let $G=(V,E)$ be a graph. Then, $p(G;i)\leq p(G;i+1)$ for $1\leq i < \frac{n}{2}$.
\end{proposition}
\proof
Let $X$ be the set of all power dominating sets of $G$ of size $i$ and $Y$ be the set of all subsets of $V$ of size $i+1$. Let $H$ be the bipartite graph with parts $X$ and $Y$, where a vertex $x\in X$ is adjacent to a vertex $y\in Y$ whenever $x\subset y$. Each $x\in X$ is adjacent to $n-i$ sets in $Y$, each of which consists of $x$ and a vertex of $G$ not in $x$. Each $y\in Y$ is adjacent to at most $i+1$ vertices in $X$, since a set of size $i+1$ has $i+1$ subsets of size $i$. Now, suppose for contradiction that for some set $S\subset X$, $|S|>|N(S)|$. Since $d(v)=n-i$ for each $v\in S$ and $|S|>|N(S)|$, there must be some vertex $w\in N(S)$ with $d(w)>n-i$. Thus, $i+1\geq d(w)>n-i$, which contradicts the assumption that $i<\frac{n}{2}$. It follows that for each $S\subset X$, $|S|\leq |N(S)|$, so by Theorem \ref{thm_hall}, $H$ has a matching that saturates all vertices of $X$. Thus, for $1\leq i < \frac{n}{2}$, to each power dominating set of size $i$, there corresponds a distinct superset of size $i+1$, which by Observation \ref{obs_subset} is also a power dominating set.
\qed
\medskip

\begin{theorem}
\label{thm_extremal_coeffs}
Let $G=(V,E)$ be a graph. Then, for $k\in \{0,\ldots,n-1\}$, 
\begin{equation*}
p(G;n-k)=\big\vert\set{S\subset V\colon|S|=k,S\cap N(V\setminus S)\text{ is a zero forcing set of $G[S]$}}\big\vert.
\end{equation*}
\end{theorem}
\proof
We will show that for any $S\subset V$, $S\cap N(V\setminus S)$ is a zero forcing set of $G[S]$ if and only if $V\setminus S$ is a power dominating set of $G$. Suppose first that $S\cap N(V\setminus S)$ is a zero forcing set of $G[S]$. If the vertices in $V\backslash S$ are initially colored, the vertices in $N(V\setminus S)$ will become colored in the domination step, and by assumption, any uncolored vertices in $S$ can be forced by the vertices in $S\cap N(V\setminus S)$. Thus $V\backslash S$ is a power dominating set of $G$. Now suppose $V\setminus S$ is a power dominating set of $G$. Then, after the domination step, the vertices in $N(V\backslash S)$ become colored, and the only uncolored vertices are in $G[S]$. Thus, in order for $V\backslash S$ to be a power dominating set, the vertices in $S\cap N(V\setminus S)$ must be able to force any uncolored vertices in $G[S]$, so $S\cap N(V\setminus S)$ is a zero forcing set of $G[S]$. Thus, for any $k\in\set{0,\ldots,n-1}$ and $S\subset V$ with $|S|=k$, $V\backslash S$ is a power dominating set of size $n-k$ if and only if $S\cap N(V\setminus S)$ is a zero forcing set of $G[S]$.
\qed

\vspace{9pt}
\noindent Note that it can be determined whether a set is a zero forcing set of a graph $G$ in linear time; thus, while computing all the coefficients of the power domination polynomial of $G$ is in general NP-hard (since, e.g., its smallest nonzero coefficient corresponds to $\gamma_P(G)$ which is NP-hard to find \cite{powerdom3}), computing $p(G;n-k)$ is in the complexity class XP, for the parameter $k$. The following are simple consequences of Theorem \ref{thm_extremal_coeffs}. 

\begin{corollary}
\label{cor_degrees}
Let $G$ be a graph with $I$ isolates and $k$ $K_2$-components. Then, 
\begin{enumerate}
\item[$1)$] $p(G;n)=1,$
\item[$2)$] $p(G;n-1)=n-I$
\item[$3)$] $p(G;n-2)=\binom{n}{2}-I(n-I)-\binom{I}{2}-k$.
\item[$4)$] If $G$ is connected and has at least 3 vertices, then $p(G;n-k)=\binom{n}{k}$ for $k\in\{0,1,2\}$.
\end{enumerate}
\end{corollary}
\proof
1) and 2) are trivial. For 3), two vertices can be excluded from a power dominating set whenever neither of them is an isolate and they do not form a $K_2$ component in $G$. 4) follows from 1), 2), and 3). 
\qed
\vspace{9pt}

\noindent Corollary \ref{cor_degrees} shows that the power domination polynomial counts the number of isolates of a graph and the number of $K_2$-components; however we will show in Proposition \ref{prop_connected} that in general, it cannot count the number of components.

\vspace{9pt}

\noindent A \emph{fort} of a graph $G=(V,E)$, as defined in \cite{caleb_thesis}, is a non-empty set $F\subset V$ such that no vertex outside $F$ is adjacent to exactly one vertex in $F$. The following facts are well-known in the literature. 

\begin{proposition}[\cite{bozeman,caleb_thesis,logan_thesis}]
\label{prop_forts_known}
Let $G=(V,E)$ be a graph and let $S\subset V$. 
\begin{enumerate}
\item $S$ is a zero forcing set of $G$ if and only if $S$ intersects every fort $F$ of $G$.  
\item $S$ is a power dominating set of $G$ if and only if $S$ intersects $N[F]$ for every fort $F$ of $G$.  
\end{enumerate}

\end{proposition}

\noindent In \cite{bozeman,logan_thesis}, it was shown that the following integer program can be used to compute the power domination number of a graph $G=(V,E)$, where $\mathcal{N}(G)=\{N[F]:F \text{ is a fort of }G\}$.   

\begin{eqnarray*}
\min&& \sum_{v\in V}s_v\\
\text{s.t.:} &&\sum_{v\in N}s_v\geq 1  \qquad\forall N\in \mathcal{N}(G)\\
&&s_v\in \{0,1\} \qquad\forall v\in V
\end{eqnarray*}

\noindent We now give a way to bound the number of constraints in this model using the power domination polynomial.

\begin{proposition}
Let $G$ be a graph. Then, $|\mathcal{N}(G)|\leq 2^n-\mathcal{P}(G;1)$.
\end{proposition}
\proof
$\mathcal{P}(G;1)=\sum_{i=1}^np(G;i)$ equals the number of distinct power dominating sets of $G$, and hence also the number of complements of power dominating sets of $G$. By Proposition \ref{prop_forts_known}, a power dominating  set must intersect every element of $\mathcal{N}(G)$. Thus, the complement of a power dominating set cannot be an element of $\mathcal{N}(G)$, and so the number of complements of power dominating sets of $G$ is at most the number of sets of $G$ which are not neighborhoods of forts, i.e., $2^n-|\mathcal{N}(G)|$. Thus, $\mathcal{P}(G;1)\leq 2^n-|\mathcal{N}(G)|$.
\qed

\medskip

\subsection{Relation to other polynomials}

In this section, we characterize the graphs whose power domination polynomials are identical to their zero forcing polynomials or their domination polynomials. We begin with the following basic observation, which follows from the fact that every zero forcing set and every dominating set of a graph is also a power dominating set.

\begin{observation}
Let $G$ be a graph. Then, $z(G;i)\leq p(G;i)$, and $d(G;i)\leq p(G;i)$ for all $1\leq i\leq n$. Also, $\gamma_P(G)\le Z(G)$ and $\gamma_P(G)\le\gamma(G)$.
\end{observation}

\noindent Note that in general, $Z(G)$ and $\gamma(G)$ (as well as $d(G;i)$ and $z(G;i)$) are incomparable. Recall also the following result relating zero forcing sets and power dominating sets.

\begin{theorem}[\cite{dean}]
\label{thm_dean}
$S$ is a power dominating set of $G$ if and only if $N[S]$ is a zero forcing set of $G$. 
\end{theorem}

\begin{theorem}
\label{thm_zf=pd}
Let $G$ be a graph. $\mathcal{Z}(G;x)=\mathcal{P}(G;x)$ if and only if $G\simeq P_{a_1}\dot\cup\cdots \dot\cup P_{a_k}$, where $a_i\leq 2$ for $1\leq i\leq k$.
\end{theorem}

\proof
Let $G$ be a graph such that $\mathcal{Z}(G;x)=\mathcal{P}(G;x)$. Suppose $G$ has a power dominating set $S$ which is not a zero forcing set. Let $S_1,\ldots,S_k$ be all zero forcing sets of $G$ of size $|S|$. Since $\mathcal{P}(G;x)=\mathcal{Z}(G;x)$, there are the same number of zero forcing sets of size $|S|$ as power dominating sets of size $|S|$. However, since $S$ is different from $S_1,\ldots,S_k$ and since every zero forcing set is also a power dominating set, it follows that $p(G;|S|)>z(G;|S|)$, a contradiction. Thus, $S$ is a zero forcing set of $G$ if and only if $S$ is a power dominating set of $G$. From this and from Theorem \ref{thm_dean}, it follows that $S$ is a zero forcing set of $G$ if and only if $N[S]$ is a zero forcing set of $G$. In other words, a dominating set of a zero forcing set of $G$ is also zero forcing.

Now, suppose some component $G'$ of $G$ has zero forcing number greater than one. Let $Z$ be a minimum zero forcing set of $G'$; $G'[Z]$ has no edges, since otherwise a dominating set of $G'[Z]$ would be a smaller zero forcing set. Let $v_1$ and $v_p$ be two vertices in $Z$ such that a shortest path $v_1,v_2,\ldots,v_p$ between $v_1$ and $v_p$ in $G'$ contains no other vertices of $Z$ (clearly such vertices exist, e.g. if they are a closest pair among all pairs of vertices in $Z$). The set $N[Z\backslash \{v_1\}\cup \{v_2\}]$ is also a zero forcing set of $G'$ since it contains $Z$; thus, $Z\backslash\{v_1\}\cup\{v_2\}$ is also a zero forcing set. Similarly, the set $N[Z\backslash \{v_1\}\cup \{v_3\}]$ is a zero forcing set of $G'$ since it contains $Z\backslash \{v_1\}\cup \{v_2\}$; hence, $Z\backslash\{v_1\}\cup\{v_3\}$ is also zero forcing. By the same reasoning, the sets $Z\backslash\{v_1\}\cup\{v_4\},\ldots,Z\backslash\{v_1\}\cup\{v_{p-1}\}$ are all zero forcing. However, $G'[Z\backslash\{v_1\}\cup\{v_{p-1}\}]$ contains the edge $v_pv_{p-1}$, and hence $G'$ has a smaller zero forcing set than $Z$, a contradiction. Thus, $Z(G')=1$ and so $G'$ is a path. If $G'$ is a path of length greater than 1, then there are subsets of $V(G')$ of size 1 which are not zero forcing sets, but all subsets of size 1 are power dominating sets. Since $\mathcal{P}(G;x)=\mathcal{Z}(G;x)$, it follows that every component of $G$ is a path of length at most 1. 

Conversely, if every component of $G$ is a path of length at most 1, then every power dominating set is clearly a zero forcing set (and vice versa), so $\mathcal{P}(G;x)=\mathcal{Z}(G;x)$.
\qed

\begin{theorem}
\label{pd=d}
Let $G=(V,E)$ be a graph. $\mathcal{D}(G;x)=\mathcal{P}(G;x)$ if and only if for each non-isolate $u\in V$, there exists a $v\in N(u)$ such that $N[v]\subset N[u]$.
\end{theorem}

\proof
Suppose that $u$ is a non-isolate of $G$ such that for each $v\in N(u)$, there is a $q\in N[v]$ with $q\notin N[u]$. Let $S=V\setminus N[u]$ be a set of initially colored vertices. Then, each vertex in $N(u)$ has a neighbor in $S$, so all vertices in $N(u)$ will be colored in the domination step. Since $u$ is the only uncolored vertex left after the domination step, and since $u$ is not an isolate, $u$ will be colored in the forcing step, so $S$ is a power dominating set of $G$. On the other hand, clearly $S$ is not a dominating set since $u$ has no neighbor in $S$. Since all dominating sets of $G$ (and in particular those of size $|S|$) are also power dominating sets, it follows that $p(G;|S|)>d(G;|S|)$ and so $\mathcal{D}(G;x)\neq\mathcal{P}(G;x)$.

Conversely, suppose that for each non-isolate $u\in V$, there exists a $v\in N(u)$ such that $N[v]\subset N[u]$, but $\mathcal{D}(G;x)\not=\mathcal{P}(G;x)$. Since every dominating set of $G$ is a power dominating set, there must be some power dominating set $S$ that is not a dominating set. Since all isolates must be in every power dominating set, there is some non-isolate $u\in V$ that is the last vertex to be forced by $S$. By our assumption, there is a $v\in N(u)$ such that $N[v]\subset N[u]$. Since $u$ is the last vertex to be colored, $v$ must have already been colored at the timestep when $u$ gets forced. If $v$ was initially colored, then $u$ would have been colored in the domination step, contradicting the fact that $u$ should be forced rather than dominated. If $v$ was colored in the domination step, then since every neighbor of $v$ is also a neighbor of $u$, $u$ would again have been dominated rather than forced, by the same vertex that dominates $v$. If $v$ was colored in some forcing step, then since every neighbor of $v$ is also a neighbor of $u$, $u$ would have had to be colored in order for the forcing to occur, contradicting the assumption that $u$ is the last vertex to be colored. Thus, in all cases, it follows that $S$ cannot exist, so $\mathcal{D}(G;x)=\mathcal{P}(G;x)$.
\qed
\vspace{9pt}

\noindent Note that the condition in Theorem \ref{pd=d} can be verified in polynomial time. Graphs satisfying this condition include, e.g., graphs where each block is a clique and contains at least two non-cut vertices.

\begin{theorem}
\label{thm_pd=zf_coeffs}
For any connected graph $G=(V,E)$, $z(G;i)=p(G;i)$ if and only if $z(G;i)=p(G;i)=\binom{n}{i}$ or $z(G;i)=p(G;i)=0$. 
\end{theorem}

\begin{proof}
One direction is trivial. For the other direction, let $G$ be a connected graph such that $z(G;i)=p(G;i)$ for some $i\in[n]$. 
Let $F$ be a fort of $G$ with minimum cardinality, and let $f=|F|$. Suppose first that $i>n-f$. Then for every $S,S'\subset V$ with $|S|=i$ and $|S'|\geq f$, $S\cap S'\neq \emptyset$. Since the size of every fort is at least $f$, every set of size $i$ will intersect every fort, and hence by Proposition \ref{prop_forts_known} will be zero forcing and power dominating. Thus $z(G;i)=p(G;i)=\binom{n}{i}$.

Now suppose $i\leq n-f$. If $z(G;i)=p(G;i)=0$, we are done. Thus suppose that $z(G;i)=p(G;i)\neq 0$. Let $Z$ be a zero forcing set of size $i$ and let $k:=|Z\cap F|$; note that since $Z$ is zero forcing, $k>0$. Then, $|V\backslash(Z\cup F)|=|V|-(|Z|+|F|-|Z\cap F|)=n-i-f+k\geq k$. If $F=V$, every set of size 1 is zero forcing and hence $z(G;i)=p(G;i)=\binom{n}{i}$. Otherwise, if $F\neq V$, since $G$ is connected, it follows that $N(F)\neq \emptyset$. Let $Z'$ be a set obtained by starting from $Z$, removing all $k$ vertices in $Z\cap F$, and adding $k$ vertices from $V\backslash (Z\cup F)$ in such a way that $Z'\cap N(F)\neq \emptyset$. Note that this is always possible, since $|V\backslash (Z\cup F)|\geq k$ and $N(F)\neq \emptyset$.
Then, by Proposition \ref{prop_forts_known}, $Z'$ is a power dominating set of size $i$, since it intersects every fort of $G$ except $F$ (because if $Z'$ does not intersect some fort $F'\neq F$, then $F'\subsetneq F$, contradicting that $F$ is minimum) and $Z'$ intersects $N[F]$. However, $Z'$ is not zero forcing since it does not intersect $F$. This contradicts the assumption that $z(G;i)=p(G;i)$.
\end{proof}
Note that the condition ``$G$ is connected" in Theorem \ref{thm_pd=zf_coeffs} is necessary, since, e.g., $0\neq p(K_3\dot\cup K_1;3)=3=z(K_3\dot\cup K_1;3)\neq \binom{4}{3}$. Moreover, the analogous statement for the domination polynomial does not hold, since, e.g., $0\neq p(S_4;1)=1=d(S_4;1)\neq \binom{4}{1}$. This fact, along with Theorems \ref{thm_zf=pd}, \ref{pd=d}, and \ref{thm_pd=zf_coeffs} shows that in general, the power domination polynomial of a graph coincides more often (both partially and completely) with its domination polynomial than with its zero forcing polynomial.

\subsection{Decomposition results}
\label{sect_decomp}
In this section, we present several results about computing the power domination polynomial of a graph in terms of the power domination polynomials of smaller graphs.

\begin{proposition}
\label{thm_disjoint_union}
If $G$ is a graph such that $G\simeq G_1\dot{\cup}G_2$, then $\mathcal{P}(G;x)=\mathcal{P}(G_1;x)\mathcal{P}(G_2;x)$.
\end{proposition}

\proof
A power dominating set of size $i$ in $G$ consists of a power dominating set of size $i_1$ in $G_1$ and a power dominating set of size $i_2=i-i_1$ in $G_2$. Since power dominating sets of size $i_1$ and $i_2$ can be chosen independently in $G_1$ and $G_2$ for each $i_1\geq \gamma_P(G_1)$, $i_2\geq \gamma_P(G_2)$, and since $p(G_1;i_1)p(G_2;i_2)=0$ for each $i_1<\gamma_P(G_1)$ or $i_2<\gamma_P(G_2)$, it follows that $p(G;i)=\sum_{i_1+i_2=i}p(G_1;i_1)p(G_2;i_2)$. The left-hand-side of this equation is the coefficient of $x^i$ in $\mathcal{P}(G;x)$, and since $\mathcal{P}(G_1;x)=\sum_{j=\gamma_P(G_1)}^{n(G_1)}p(G_1;j)x^j$ and $\mathcal{P}(G_2;x)=\sum_{j=\gamma_P(G_2)}^{n(G_2)}p(G_2;j)x^j$, the right-hand-side of the equation is the coefficient of $x^i$ in $\mathcal{P}(G_1;x)\mathcal{P}(G_2;x)$. Thus, $\mathcal{P}(G_1;x)\mathcal{P}(G_2;x)$ and $\mathcal{P}(G;x)$ have the same coefficients and the same degree, so they are identical. 
\qed

\begin{corollary}
\label{cor_isolate}
Let $G$ be a graph on $n$ vertices. Then $p(G\dot\cup K_1;i)=p(G;i-1)$ for each $i\in[n+1]$, and  $\mathcal{P}(G\dot\cup \overline{K_k};x)=x^k\mathcal{P}(G;x)$ for all $k\in \mathbb{N}$.
\end{corollary}

\begin{lemma}
\label{lemma_identity}
For any $n_1,n_2\in \mathbb{N}$,
\begin{equation*}
\sum_{i=1}^{n_1+n_2}\sum\limits_{\substack{i_1,i_2\in\mathbb{N}\\i_1+i_2=i}}\binom{n_1}{i_1}\binom{n_2}{i_2}x^i=((x+1)^{n_1}-1)((x+1)^{n_2}-1).
\end{equation*} 
\end{lemma}
\proof
Let $A$ and $B$ be disjoint sets with $|A|=n_1$ and $|B|=n_2$. The number of ways to choose a subset of $A\cup B$ of size $i$ which intersects both $A$ and $B$ is 

\begin{equation*}
\sum\limits_{\substack{i_1,i_2\in\mathbb{N}\\i_1+i_2=i}}\binom{n_1}{i_1}\binom{n_2}{i_2}.
\end{equation*}
Counting another way, this quantity is equal to 

\begin{equation*}
\binom{n_1+n_2}{i}-\binom{n_1}{i}-\binom{n_2}{i},
\end{equation*}
because $\binom{n_1}{i}$ of the $\binom{n_1+n_2}{i}$ subsets of $A\cup B$ of size $i$ contain vertices only from $A$, and $\binom{n_2}{i}$ contain vertices only from $B$. Then,
\begin{align*}
\sum_{i=1}^{n_1+n_2}\sum\limits_{\substack{i_1,i_2\in\mathbb{N}\\i_1+i_2=i}}\binom{n_1}{i_1}\binom{n_2}{i_2}x^i&=
\sum_{i=1}^{n_1+n_2}\left(\binom{n_1+n_2}{i}-\binom{n_1}{i}-\binom{n_2}{i}\right)x^i\\
&=\sum_{i=1}^{n_1+n_2}\binom{n_1+n_2}{i}x^i-\sum_{i=1}^{n_1}\binom{n_1}{i}x^i-\sum_{i=1}^{n_2}\binom{n_2}{i}x^i\\
&=((x+1)^{n_1+n_2}-1)-((x+1)^{n_1}-1)-((x+1)^{n_2}-1)\\
&=((x+1)^{n_1}-1)((x+1)^{n_2}-1),
\end{align*}
where the second and third equalities follow from the convention that ${a\choose b}=0$ if $a<b$ and from the binomial theorem.
\qed

\begin{theorem}
	\label{thm_joins}
Let $G_1$ and $G_2$ be graphs on $n_1$ and $n_2$ vertices, respectively. For $j\in \{1,2\}$, let $I_j$ equal the number of isolates of $G_j$ if $n_j>1$, and equal zero if $n_j=1$. 
Let $G=G_1\vee G_2$. Then,
\begin{equation*}
\mathcal{P}(G;x)=(1+I_1/x)\mathcal{P}(G_1;x)+(1+I_2/x)\mathcal{P}(G_2;x)+((x+1)^{n_1}-1)((x+1)^{n_2}-1).
\end{equation*}	
\end{theorem}

\proof
The power dominating sets of $G$ can be partitioned into those which intersect both $V(G_1)$ and $V(G_2)$, and those which intersect only one of $V(G_1)$ or $V(G_2)$. Since each vertex of $G_1$ is adjacent to each vertex of $G_2$ and vice versa, any set of vertices which intersects both $V(G_1)$ and $V(G_2)$ is a power dominating set of $G$. Moreover, there are 


\begin{equation*}
\sum\limits_{\substack{i_1,i_2\in\mathbb{N}\\i_1+i_2=i}}\binom{n_1}{i_1}\binom{n_2}{i_2}
\end{equation*} 
such sets of size $i$, for any $i\in [n_1+n_2]$. For $j\in \{1,2\}$, a set $S\subset V(G_j)$ is a power dominating set of $G$ if and only if one of the following conditions holds: 

\begin{enumerate}
\item $S$ is a power dominating set of $G_j$.
\item $I_j\geq 1$ and $S$ is a power dominating set of $G_j-v$, where $v$ is an isolate of $G_j$.
\end{enumerate}
Note that in the second case, $S$ is a power dominating set because $V(G)\backslash V(G_j)$ will be colored in the domination step by some vertex other than $v$, then $N[S]\cap V(G_j)$ will force any uncolored vertices of $G_j-v$, and finally $v$ can be forced by any vertex in $V(G)\backslash V(G_j)$. For any $i\in [n_1+n_2]$, there are $p(G_j;i)$ power dominating sets of $G_j$ of size $i$, and $p(G_j;i+1)$ power dominating sets of $G_j-v$ of size $i$ for each isolate $v$ of $G_j$ (recall that $p(G_j;k)=0$ if $k>n_j$). Therefore,

\begin{eqnarray*}
\mathcal{P}(G;x)&=&\sum\limits_{i=1}^{n_1+n_2}\left(\sum\limits_{\substack{i_1,i_2\in\mathbb{N}\\i_1+i_2=i}}\binom{n_1}{i_1}\binom{n_2}{i_2}+p(G_1;i)+I_1p(G_1;i+1)+p(G_2;i)+I_2p(G_2;i+1)\right)x^i\\
&=&((x+1)^{n_1}-1)((x+1)^{n_2}-1)+\mathcal{P}(G_1;x)+\mathcal{P}(G_2;x)+\\
&&+I_1(\mathcal{P}(G_1;x)/x-p(G_1;1))+I_2(\mathcal{P}(G_2;x)/x-p(G_2;1))\\
&=&(1+I_1/x)\mathcal{P}(G_1;x)+(1+I_2/x)\mathcal{P}(G_2;x)+((x+1)^{n_1}-1)((x+1)^{n_2}-1),
\end{eqnarray*}
where the second equality follows from Lemma \ref{lemma_identity}, and the last equality follows from the fact that, for $j\in \{1,2\}$, $I_jp(G_j;1)=0$ since if $I_j>0$, then then $G_j$ has at least two connected components, so $p(G_j;1)=0$. \qed

\begin{corollary}
\label{cor_dom_vertex}
Let $G$ be a graph on $n$ vertices, and 
$I$ equal the number of isolates of $G$ if $n>1$, and equal zero if $n=1$. Let $G'$ be the graph obtained from $G$ by adding a dominating vertex. Then $\mathcal{P}(G';x)=(1+I/x)\mathcal{P}(G;x)+x(x+1)^{n}$.
\end{corollary}
\vspace{9pt}

\begin{theorem}
\label{theorem_corona}
Let $H$ be a graph with vertex set $\{v_1,\ldots,v_n\}$. For $1\leq i\leq n$, let $G_i$ be a graph of order $n_i$ with $p(G_i;1)=n_i$, and let $u_i$ be any vertex of $G_i$, except, if $G_i$ is a path, $u_i$ is not an endpoint of the path. Let $G$ be the graph obtained by identifying $u_i$ and $v_i$ for $1\leq i\leq n$. Then,
\begin{equation*}
\mathcal{P}(G;x)=\prod_{i=1}^n((x+1)^{n_i}-1).
\end{equation*}
\end{theorem}

\proof
Let $G'=\dot\bigcup_{i=1}^n G_i$. We will show that $S$ is a power dominating set of $G'$ if and only if $S$ is a power dominating set of $G$. By Propositions \ref{prop_all_sets} and \ref{thm_disjoint_union}, it will follow that $\mathcal{P}(G;x)=\prod_{i=1}^n\left(\sum_{j=1}^{n_i}\binom{n_i}{j}x^j\right)=\prod_{i=1}^n((x+1)^{n_i}-1)$.

Let $S$ be a power dominating set of $G'$. Then $S$ contains at least one vertex of $V(G_i)$, $1\leq i\leq n$. Moreover, for $1\leq i\leq n$, since $u_i$ is the only vertex of $G[V(G_i)]$ which may have neighbors outside $G[V(G_i)]$ in $G$, $S\cap V(G_i)$ will power dominate in $G[V(G_i)]$ as it does in $G'[V(G_i)]$ until $u_i$ is colored, at which point the forcing could possibly stop. However, since forcing will proceed in each $G[V(G_i)]$, all vertices $u_i$ in $G$ will eventually be colored by the respective $S\cap V(G_i)$. Then, all neighbors of each $u_i$ outside $G[V(G_i)]$ will be colored, and forcing will resume in each $G[V(G_i)]$ as in $G'[V(G_i)]$. Thus, $S$ will power dominate all of $G$.

Now let $S$ be a power dominating set of $G$. Suppose there is some $j\in [n]$ for which $S\cap V(G_j)=\emptyset$. Since $S\subset V(G)\backslash V(G_j)$, by Observation \ref{obs_subset}, $V(G)\backslash V(G_j)$ is a power dominating set of $G$. Then, by Theorem \ref{thm_dean}, $N[V(G)\backslash V(G_j)]\cap V(G_j)$ is a zero forcing set of $G[V(G_j)]$. If $v_j$ is an isolate in $H$, then $N[V(G)\backslash V(G_j)]\cap V(G_j)=\emptyset$, so $S$ cannot be a power dominating set of $G$, a contradiction. Otherwise, $N[V(G)\backslash V(G_j)]\cap V(G_j)=\{u_j\}$. However, the only graph for which a single vertex is a zero forcing set is a path, and this happens only when the vertex is an endpoint of the path. This contradicts our assumption that if $G_i$ is a path, $u_i$ is not an endpoint of the path. Thus, $S\cap V(G_j)\neq\emptyset$, so $S$ contains at least one vertex $x_i$ of $V(G_i)$ for all $i\in [n]$. Then, since $p(G_i;1)=n_i$, $\{x_i\}$ is a power dominating set of $G'[V(G_i)]$ for all $i$, so $S$ is a power dominating set of $G'$.
\qed

\vspace{9pt}

\section{Characterizing $\mathcal{P}(G;x)$ for specific graphs}
\label{section_characterizations}
In this section, we give closed-form expressions and algorithms to compute the power domination polynomials of several families of graphs, as well as some characterizations of uniqueness. We have also implemented a brute force algorithm for computing the power domination polynomials of arbitrary graphs (cf. \url{https://github.com/rsp7/Power-Domination-Polynomial}), and used it to compute the power domination polynomials of all graphs on fewer than 10 vertices.\\

\begin{proposition}
	\label{prop_closed_form1}
	\begin{enumerate}
		\item[]
		\item $\mathcal{P}(K_n;x)=\mathcal{P}(P_n;x)=\mathcal{P}(C_n;x)=\mathcal{P}(W_n;x)=(x+1)^n-1$.
		\item $\mathcal{P}(\overline{K_n};x)=x^n$.
		\item $\mathcal{P}(S_n;x)=x(x+1)^{n-1}+x^{n-1}+(n-1)x^{n-2}$, for $n\geq 3$.
		\item $\mathcal{P}(H\circ K_k;x)=((x+1)^{k+1}-1)^n$ for any graph $H$ of order $n$, and $k>1$. 
	\end{enumerate}
\end{proposition}

\proof
\begin{enumerate}
	\item[]
	\item It is easy to see that any vertex of $K_n$, $P_n$, $C_n$, or $W_n$ forms a power dominating set of size 1. Then, by Observation \ref{obs_subset}, every set of vertices of these graphs is a power dominating set. It follows that $\mathcal{P}(K_n;x)=\mathcal{P}(P_n;x)=\mathcal{P}(C_n;x)=\mathcal{P}(W_n;x)=\sum_{i=1}^n\binom{n}{i}x^i=(x+1)^n-1$.
	\item The only power dominating set of $\overline{K_n}$ is the set containing all of its vertices, so $\mathcal{P}(\overline{K_n};x)=x^n$.
	\item Since $S_n\simeq \overline{K_{n-1}}\lor K_1$, by Corollary \ref{cor_dom_vertex} it follows that $\mathcal{P}(S_n;x)=(1+(n-1)/x)\mathcal{P}(\overline{K_{n-1}};x)+x(x+1)^{n-1}=x^{n-1}+(n-1)x^{n-2}+x(x+1)^{n-1}$.
	\item This follows by Theorem \ref{theorem_corona}, where $G_i\simeq K_{k+1}$ for $1\leq i\leq n$. Note that since $k>1$, each $(k+1)$-clique whose vertex is identified with a vertex of $H$ is different from a path (i.e., it is not $K_1$ or $K_2$).
	\qed
\end{enumerate}
\vspace{9pt}

\noindent A graph $G=(V,E)$ is a \emph{threshold graph} if there exists a real number $t$ and a function $w:V\to \mathbb{R}$ such that $uv\in E$ if and only if $w(u)+w(v)\geq t$. For a binary string $B$, the \emph{threshold graph generated by $B$}, written $T(B)$, is the graph whose vertices are the symbols in $B$, and which has an edge between a pair of symbols $x$ and $y$ with $x$ to the left of $y$ if and only if $y=1$. It was shown by Chv{\'a}tal and Hammer \cite{chvatal} that every threshold graph is generated by some binary string, and that for a particular threshold graph, this string is unique apart from the first symbol (since changing the first symbol  from 0 to 1, or from 1 to 0, would not affect $T(B)$). Thus we will refer to $B$ and $T(B)$ interchangeably, where symbols of $B$ correspond to vertices in $T(B)$. 

A \emph{block} of a binary string $B$ is a maximal contiguous substring consisting either of only 0s or only 1s. The partition of $B$ into its blocks is called the \emph{block partition} of $B$, and when the block partition of $B$ has $\omega$ blocks, we label the blocks $B_{i}$, $1\leq i\leq \omega$, and write $B = B_{1}B_{2}\dots B_{\omega}$. Each $B_{i}$ consisting of 0s is called a 0-block, and each $B_{i}$ consisting of 1s is a 1-block. Similarly, we will call a vertex in a 0-block a 0-vertex, and a vertex in a 1-block a 1-vertex. We will refer to $|B_i|$ as $b_i$.     We will assume that $T(B)$ is connected, i.e. that $B_{\omega}$ is a 1-block. Note that if $B_{\omega}$ is a 1-block, then by Corollary \ref{cor_isolate}, $\mathcal{P}(T(B_1 B_2\ldots B_{\omega} B_{\omega+1});x)=x^{b_{\omega+1}}\mathcal{P}(T(B_1 B_2\ldots B_{\omega});x)$.
We will also assume without loss of generality that $B$ has at least two symbols, and that the first symbol of $B$ is the same as the second symbol, so that $b_1\geq 2$. We now give a characterization of the power dominating sets of a threshold graph, and an algorithm to efficiently compute the coefficients of the power domination polynomial.
\newpage

\begin{theorem}
Let $T = T(B)$ be a threshold graph generated by a binary string $B$. A set $S \subset V(T)$ is a power dominating set of $T$ if and only if there exists $v \in S$ such that
	\begin{enumerate}
		\item[$a)$]  All 1-vertices are in $N[v]$.
		\item[$b)$]  All 0-blocks following or including the block containing $v$ have at most one vertex not in $S$.
	\end{enumerate}
\end{theorem}

\begin{proof}
Suppose conditions $a)$ and $b)$ are true. By condition $a)$, all 1-vertices are in $N[v]$ and will be colored in the domination step. Let $B_i$ be the block with the smallest index which contains an uncolored vertex $u$ after the domination step. Since all 1-vertices are colored, $B_i$ is a 0-block; by condition $b)$, $u$ is the only uncolored vertex in $B_i$. Then, any 1-vertex in $B_{i+1}$ can force $u$. This process can be repeated until all vertices are colored; thus, $S$ is a power dominating set.

Now suppose condition $a)$ is false. Since all 1-vertices are adjacent to each other, $S$ contains no 1-vertices; moreover, since for each vertex $v$ in $B_1$, all 1-vertices are in $N[v]$, $S$ does not contain any vertices from $B_1$. The vertices in $B_1$ are therefore not in $S$ and not adjacent to any vertex in $S$, so they can only be colored in a forcing step. However, since any vertex not in $B_1$ is adjacent to all or none of the vertices in $B_1$, and since $|B_1| \geq 2$, these vertices can never be forced.

Finally, suppose condition $a)$ is true and $b)$ is false. Then, for any $v\in S$ satisfying condition $a)$, some 0-block following or including the block containing $v$ has at least two vertices not in $S$. Let $B_i$ be the block with the largest index which contains a vertex of $S$ satisfying condition $a)$. 
$B_i\neq B_{\omega}$, since then condition $b)$ would be true. Thus, some 0-block $B_j$ with $j\geq i$ has at least two vertices not in $S$, and $S$ does not contain vertices of any 1-block $B_k$ with $k>j$. Then, the vertices in $B_j$ which are not in $S$ cannot be colored by the domination step; they also cannot be forced, since any vertex which is adjacent to one of them is adjacent to all of them. 
\end{proof}

\begin{algorithm2e}[h!]
\textbf{Input:} Binary string $B=B_1\ldots B_\omega$\;
\textbf{Output:} Values $a_1,\ldots,a_n$ such that $\mathcal{P}(T(B);x)=\sum_{i=1}^na_ix^i$\;
\For{$i=1$ \emph{\textbf{to}} $n$}{
$a_i\leftarrow 0$\;
}
\If{$B_1$ \emph{is a 1-block}}{
\For{$j=1$ \emph{\textbf{to}} $b_1$}{
$a_j\leftarrow \binom{b_1}{j}$\;
}
$i\leftarrow 2$\;
}
\If{$B_1$ \emph{is a 0-block}}{
$a_{b_1}\leftarrow 1$\;
$a_{b_1-1}\leftarrow b_1$\;
\For{$j=1$ \emph{\textbf{to}} $b_1+b_2$}{
$a_j\leftarrow a_j+\binom{b_1+b_2}{j}$\;
}
\For{$j=1$ \emph{\textbf{to}} $b_1$}{
$a_j\leftarrow a_j-\binom{b_1}{j}$\;
}
$i\leftarrow 3$;
}
\While{$i\leq \omega-1$}{
$s\leftarrow b_1+\ldots+b_{i-1}$\;
\For{$j=s+b_i$ \emph{\textbf{to}} $b_i+1$}{
$a_j\leftarrow a_{j-b_i}$\;
}

\For{$j=1$ \emph{\textbf{to}} $b_i$}{
$a_j\leftarrow 0$\;
}
$i\leftarrow i+1$\;
$s\leftarrow b_1+\ldots+b_{i-1}$\;
\For{$j=1$ \emph{\textbf{to}} $s-1$}{
$a_j\leftarrow a_j+a_{j+1}b_{i-1}$\;
}
\For{$j=1$ \emph{\textbf{to}} $s+b_i$}{
$a_j\leftarrow a_j+\binom{s+b_i}{j}$\;
}
\For{$j=1$ \emph{\textbf{to}} $s$}{
$a_j\leftarrow a_j-\binom{s}{j}$\;
}
$i\leftarrow i+1$\;
}
\Return{$a_1,\ldots,a_n$}\;
\caption{Finding coefficients of $\mathcal{P}(T(B);x)$}

\label{alg_thr}
\end{algorithm2e}

\begin{theorem}
Let $T = T(B)$ be a threshold graph generated by a binary string $B$. The coefficients of $\mathcal{P}(T;x)$ can be computed in $O(n^2)$ time with Algorithm \ref{alg_thr}.
\end{theorem}

\proof
Let $B=B_1 B_2 \ldots B_\omega$ and for $i\in [\omega]$, let $T_i=T(B_1 B_2 \ldots B_i)$; thus, $T=T_\omega$. For $2\leq i\leq \omega$, if $B_i$ is a 0-block, then $T_i=T_{i-1}\dot\cup \overline{K_{b_i}}$, so by Proposition \ref{thm_disjoint_union},
\begin{equation}
\label{eq_thr_1}
\mathcal{P}(T_i;x)=x^{b_i}\mathcal{P}(T_{i-1};x).
\end{equation}
For $2\leq i\leq \omega$, if $B_i$ is a 1-block, then $T_i=T_{i-1}\lor K_{b_i}$, so by Theorem \ref{thm_joins} (and since by assumption, $b_1>1$),
\begin{equation}
\label{eq_thr_2}
\mathcal{P}(T_i;x)=\mathcal{P}(T_{i-1};x)+\frac{b_{i-1}}{x}\mathcal{P}(T_{i-1};x)+(x+1)^{b_1+\ldots+b_i}-(x+1)^{b_1+\ldots+b_{i-1}}.
\end{equation}
Algorithm \ref{alg_thr} begins with an all-zero array of coefficients for $\mathcal{P}(T;x)$. In the first if-statement, if $B_1$ is a 1-block, the coefficients of $\mathcal{P}(T_1;x)$ are computed by the binomial expansion of $(x+1)^{b_1}$. In the second if-statement, if $B_1$ is a 0-block, by our assumption $T$ has at least one 1-block; thus, $\mathcal{P}(T_2;x)$ is computed by first setting ``$a_{b_1}\leftarrow 1$" for $\mathcal{P}(T_1;x)=x^{b_1}$, and then setting ``$a_{b_1-1}\leftarrow b_1$" for $\frac{b_1}{x}\mathcal{P}(T_1;x)$, adding the binomial expansion of $(x+1)^{b_1+b_2}$, and subtracting the binomial expansion of $(x+1)^{b_1}$ as in \eqref{eq_thr_2}.

After the two initial if-statements, the while-loop evaluates $\mathcal{P}(T_i;x)$ and $\mathcal{P}(T_{i+1};x)$ for $i\leq \omega-1$, starting with $i$ equal to either 2 or 3 depending on whether $B_1$ was a 1-block or a 0-block; in either case, the first block (if any) to be evaluated by the while-loop is a 0-block, followed by a 1-block. The first two for-loops in the while-loop shift the indices of the coefficients of $\mathcal{P}(T_{i-1};x)$ to the right by $b_i$, which is equivalent to multiplying $\mathcal{P}(T_{i-1};x)$ by $x^{b_i}$ to obtain $\mathcal{P}(T_i;x)$ as in \eqref{eq_thr_1}. Then, $i$ is incremented as a 1-block is to be added. The third for-loop adds to each coefficient $a_j$ the coefficient to the right of $a_j$ multiplied by $b_{i-1}$; this is equivalent to adding $\frac{b_{i-1}}{x}\mathcal{P}(T_{i-1};x)$ to $\mathcal{P}(T_{i-1};x)$ as in \eqref{eq_thr_2}. The last two for-loops respectively add the binomial expansion of $(x+1)^{b_1+\ldots+b_i}$ and subtract the binomial expansion of $(x+1)^{b_1+\ldots+b_{i-1}}$, as in \eqref{eq_thr_2}. Thus, by the end of the last for-loop, $\mathcal{P}(T_i;x)$ is computed from $\mathcal{P}(T_{i-1};x)$ as in \eqref{eq_thr_2}, and then $i$ is incremented. When $i=\omega+1$, the while-loop terminates, and the last computed polynomial is $\mathcal{P}(T_\omega;x)=\mathcal{P}(T;x)$; its coefficients $a_1,\ldots,a_n$ are returned by the algorithm.

To verify the runtime, note that the first for-loop and the two if-statements in Algorithm \ref{alg_thr} can each be evaluated in $O(n)$ time; the while-loop executes $O(\omega)=O(n)$ times, with each for-loop inside it taking $O(n)$ time. Thus, the total runtime is $O(n^2)$.
\qed

\subsection{$\mathcal{P}$-unique graphs}

In this section, we identify several families of graphs which can be recognized by their power domination polynomials; more precisely, we identify families $\mathcal{F}$ such that if $G\in\mathcal{F}$, then $\mathcal{P}(H;x)=\mathcal{P}(G;x)$ implies $H\simeq G$. Let us call families of graphs satisfying this property \emph{$\mathcal{P}$-unique}. We also identify arbitrarily large sets of graphs which all have the same power domination polynomial and show that the power domination polynomial is generally not effective at measuring vertex connectivity.

\begin{theorem}
\label{thm_isolate_unique}
Let $G$ be a graph.
\begin{enumerate}
\item $G$ is $\mathcal{P}$-unique if and only if $G\dot\cup K_1$ is $\mathcal{P}$-unique.
\item $G$ is $\mathcal{P}$-unique if and only if $G\dot\cup K_2$ is $\mathcal{P}$-unique.
\item For any $k\geq 3$, $G\dot\cup K_k$ is not $\mathcal{P}$-unique.
\end{enumerate}
\end{theorem}

\proof
Note that for any graphs $G_1$, $G_2$, and $G_3$, $G_1\simeq G_2$ if and only if $G_1\dot\cup G_3\simeq G_2\dot\cup G_3$. Let $K$ be either the graph $K_1$ or $K_2$.

Suppose that $G$ is $\mathcal{P}$-unique, and let $H$ be a graph such that $\mathcal{P}(G\dot\cup K;x)=\mathcal{P}(H;x)$. By Corollary \ref{cor_degrees}, $G\dot\cup K$ and $H$ have the same number of $K$-components; in particular, $H=H'\dot\cup K$ for some $H'$. Using Proposition \ref{thm_disjoint_union}, we have that $\mathcal{P}(G;x)=\mathcal{P}(G\dot\cup K;x)/\mathcal{P}(K;x)=\mathcal{P}(H;x)/\mathcal{P}(K;x)=\mathcal{P}(H';x)$. Since $G$ is $\mathcal{P}$-unique, this implies that $G\simeq H'$ and hence that $G\dot\cup K\simeq H$, which means that $G\dot\cup K$ is also $\mathcal{P}$-unique.

Suppose $G\dot\cup K$ is $\mathcal{P}$-unique and let $H$ be a graph such that $\mathcal{P}(G;x)=\mathcal{P}(H;x)$. By Proposition \ref{thm_disjoint_union}, we have that $\mathcal{P}(G\dot\cup K;x)=\mathcal{P}(G;x)\mathcal{P}(K;x)=\mathcal{P}(H;x)\mathcal{P}(K;x)=\mathcal{P}(H\dot\cup K;x)$. Since $G\dot\cup K$ is $\mathcal{P}$-unique, this implies that $G\dot\cup K\simeq H\dot\cup K$ and hence that $G\simeq H$, which means that $G$ is also $\mathcal{P}$-unique.

By Propositions \ref{thm_disjoint_union} and \ref{prop_closed_form1}, for $k\geq 3$, $\mathcal{P}(G\dot\cup K_k)=\mathcal{P}(G\dot\cup P_k)$ but $G\dot\cup K_k\not\simeq G\dot\cup P_k$.\qed

\begin{corollary}
\begin{enumerate}
\item[]
\item $\overline{K_n}$ is $\mathcal{P}$-unique.
\item $\dot{\bigcup}_{i=1}^kK_2$ is $\mathcal{P}$-unique, for all $k\in\mathbb{N}$.
\item $\left(\dot{\bigcup}_{i=1}^kK_2\right)\dot\cup\overline{K_\ell}$ is $\mathcal{P}$-unique, for all $k,\ell\in\mathbb{N}$.
\end{enumerate}
\end{corollary}
\newpage

\begin{theorem}
\label{thm_star_unique}
$S_n$ is $\mathcal{P}$-unique for $n\geq 4$.
\end{theorem}
\proof
Since $\gamma_P(S_n)=1$, it suffices to show that if $H$ is a graph on $n$ vertices such that $\gamma_P(H)=1$ and $H\not\simeq S_n$, then $\mathcal{P}(H;x)$ differs from $\mathcal{P}(S_n;x)$ in at least one coefficient. Let $\{v\}$ be a power dominating set of $H$. Since any power dominating set of $H$ must contain at least one vertex from each connected component, $H$ is connected and hence $v$ has at least one neighbor. If $v$ has exactly one neighbor, then the fact that $\{v\}$ is power dominating implies that $H\simeq P_n$, in which case $p(H;1)=n>p(S_n;1)$ by Proposition \ref{prop_closed_form1}.

Hence, suppose that $v$ has at least two neighbors. If some neighbor $u$ of $v$ is adjacent to a vertex $w$ that $v$ is not adjacent to, then $V(H)\setminus\{u,v,w\}$ is a power dominating set of $H$, since some other neighbor of $v$ will dominate $v$, and then $v$ can force $u$ and $u$ can force $w$. Otherwise, since $n\ge4$ and $H$ is connected but not a star, $v$ has at least three neighbors, two of which, say $u$ and $w$, are adjacent. If $x$ is a third neighbor of $v$, then $V(H)\setminus\{u,v,x\}$ is a power dominating set of $H$, since $w$ can dominate $v$ and $u$, and $v$ can force $x$. In each of these two cases, each of the $\binom{n-1}{n-4}$ subsets of $V(H)$ including $v$ and excluding any three vertices is power dominating by Observation \ref{obs_subset}. Note that $p(S_n;n-3)=\binom{n-1}{n-4}$ by Proposition \ref{prop_closed_form1}, so $p(H;n-3)>p(S_n;n-3)$. Thus,  in all cases, $\mathcal{P}(H;x)\neq \mathcal{P}(S_n;x)$.
\qed
\vspace{9pt}

\noindent Note that the condition ``$n\geq 4$" in Theorem \ref{thm_star_unique} is necessary, since for the degenerate case $n=3$, $\mathcal{P}(S_3;x)=\mathcal{P}(K_3;x)$ but $S_3\not\simeq K_3$.

\begin{proposition}
\label{prop_connected}
\begin{enumerate}
\item[]
\item There exist arbitrarily large sets of graphs that all have the same power domination polynomial.
\item A connected and a disconnected graph can have the same power domination polynomial.
\item Graphs with vertex connectivity $1$, $2$, and $n-1$ can have the same power domination polynomial.
\end{enumerate}
\end{proposition}
\proof

\begin{enumerate}
\item[]
\item Adding a single chord to $C_n$, $n\ge 4$, produces a graph in which any set of size 1 is a power dominating set. Thus, by Observation \ref{obs_subset} and Proposition \ref{prop_closed_form1}, all such graphs have the same power domination polynomial as $C_n$. Since the number of distinct (up to isomorphism) ways to add a chord to $C_n$ is $\lfloor n/2\rfloor-1$, this yields arbitrarily large sets of graphs with the same power domination polynomial.
\item Let $G$ be the graph obtained by connecting two cycles by a single edge $e$. $S$ is a power dominating set of $G$ if and only if it contains at least one vertex from each cycle; the same is true for $G-e$. Thus, $\mathcal{P}(G;x)=\mathcal{P}(G-e;x)$.
\item By Proposition \ref{prop_closed_form1}, $P_n$, $C_n$, and $K_n$ all have the same power domination polynomial.\qed
\end{enumerate}

\section{Roots of power domination polynomials}
\label{section_roots}

We define a \emph{power domination root} of a graph $G$ to be a root of $\mathcal{P}(G;x)$. In this section, we study various properties of power domination roots. We begin with the following basic facts.

\begin{proposition}
\label{thm_roots}
Let $G$ be a graph. Then:
\begin{enumerate}
	\item Zero is a power domination root of $G$ of multiplicity $\gamma_P(G)$.
	\item $G$ cannot have positive power domination roots.
	\item $G$ may have complex power domination roots.
	\item If $r$ is a real rational power domination root of $G$, then $r$ is an integer.
\end{enumerate}
\end{proposition}
\proof
\begin{enumerate}
\item[]
\item This follows from the fact that the first nonzero coefficient of $\mathcal{P}(G;x)$ corresponds to $x^{\gamma_P(G)}$.
	\item $\mathcal{P}(G;x)$ is a non-constant polynomial with nonnegative coefficients, and hence the derivative of $\mathcal{P}(G;x)$ is a nonzero polynomial with nonnegative coefficients. Thus $\mathcal{P}(G;x)$ is strictly increasing on $(0,\infty)$, and $\mathcal{P}(G;0)=0$, so it cannot have positive roots.
\item By Proposition \ref{prop_closed_form1}, $K_n$, $n \geq 3$, has complex power domination roots, since there exist complex $n^\text{th}$ roots of unity for $n\geq 3$. For example, $\mathcal{P}(K_4;x)=x^4+4x^3 + 6x^2 + 4x$ has complex roots $-1\pm i$.
	\item This follows from the Rational Root Theorem and the fact that $p(G;n)=1$.\qed

\end{enumerate}

\vspace{9pt}
\noindent We will now characterize graphs having a small number of distinct power domination roots. The next observation follows immediately from Propositions \ref{prop_closed_form1} and \ref{thm_roots}.

\begin{observation}
\label{obs_roots3}
A graph $G$ has exactly one distinct power domination root if and only if $G\simeq \overline{K_n}$.
\end{observation}

\begin{definition}[\cite{powerdom2}]
\label{def_zhao}
Let $\mathscr{F}$ be the family of graphs obtained from connected graphs $H$ by adding two new vertices $v'$ and $v''$, two new edges $vv'$ and $vv''$, and possibly the edge $v'v''$, for each vertex $v$ of $H$.	
\end{definition}

\begin{theorem}[\cite{powerdom2}]
\label{thm_zhao}
If $G=(V,E)$ is a connected graph of order $n\ge3$, then $\gamma_P(G)\le n/3$ with equality if and only if $G\in\mathscr{F}\cup\{K_{3,3}\}$.
\end{theorem}

\begin{theorem}
\label{thm_two_roots}
A graph $G$ has exactly two distinct power domination roots if and only if $G \simeq G_1\dot\cup\ldots\dot\cup G_k\dot\cup\overline{K_r}$, where $k\geq 1$, $r\geq 0$, and $G_i\simeq P_2$ for $1\leq i\leq k$. Moreover, if $G$ has exactly two power domination roots, these roots are $0$ and $-2$.
\end{theorem}
\proof
If $G \simeq G_1\dot\cup\ldots\dot\cup G_k\dot\cup\overline{K_r}$, where $k\geq 1$, $r\geq 0$, and $G_i\simeq P_2$ for $1\leq i\leq k$, then by Proposition \ref{thm_disjoint_union} and Proposition \ref{prop_closed_form1}, $\mathcal{P}(G;x)=x^r(x^2+2x)^k$; thus, $G$ has exactly two distinct power domination roots: $0$ and $-2$. 

Now, suppose that $G$ has exactly two distinct power domination roots. Let $G'$ be the largest component of $G$ (by number of vertices), and let $n'=|V(G')|$. Clearly  $n'>1$, since otherwise $G\simeq \overline{K_n}$ and $G$ would have only one power domination root by Observation \ref{obs_roots3}.

Suppose for contradiction that $n'\ge3$. Since $0$ is a power domination root of $G'$ of multiplicity $\gamma_P(G')$, we have that $\mathcal{P}(G';x)=x^p(x+a)^{n'-p}$ for some $p\in[n'-1]$ and $a>0$ (since imaginary roots appear in complex conjugate pairs, and there cannot be any positive roots by Proposition \ref{thm_roots}). Note that the coefficient of $x^{n'-1}$ in $x^p(x+a)^{n'-p}$ is $(n'-p)a$. Since $G'$ is connected and has more than one vertex, it has no isolates, so it follows from Corollary \ref{cor_degrees} that $(n'-p)a=n'$. Since $n'$ and $p$ are integers, $a$ is rational, and by Proposition \ref{thm_roots}, $a$ is an integer. Moreover, since $p>0$, we have that $a\ge2$. Since the lowest-order term in $\mathcal{P}(G';x)$ involves $x^p$, we have that $\gamma_P(G')=p$. By Theorem \ref{thm_zhao}, $p=\frac{n'(a-1)}{a}\le \frac{n'}{3}$, which implies that $a\le3/2$, a contradiction. Thus it cannot hold that $n'\geq 3$, so $n'=2$. Then, $G \simeq G_1\dot\cup\ldots\dot\cup G_k\dot\cup\overline{K_r}$, for some $k\geq 1$, $r\geq 0$, and $G_i\simeq P_2$ for $1\leq i\leq k$.
\qed

\vspace{9pt}

\begin{theorem}
\label{thm_3_roots}
A graph $G$ has exactly three distinct power domination roots if and only if $G \simeq G_1\dot\cup\ldots\dot\cup G_k\dot\cup\overline{K_r}$, where $k\geq 1$, $r\geq 0$, and $G_i\in \mathscr{F}$ for $1\leq i\leq k$. Moreover, if $G$ has exactly three power domination roots, these roots are $0$, $\frac{-3+\sqrt3 i}{2}$, and $\frac{-3-\sqrt 3 i}{2}$.
\end{theorem}
\proof
If $G \simeq G_1\dot\cup\ldots\dot\cup G_k\dot\cup\overline{K_r}$, where $k\geq 1$, $r\geq 0$, and $G_i\in \mathscr{F}$ for $1\leq i\leq k$, then $G_1\dot\cup\ldots\dot\cup G_k$ can be viewed as a graph obtained from a graph $H$ by identifying a vertex of $K_3$ or a non-endpoint vertex of $P_3$ to each vertex of $H$. Thus, $G_1\dot\cup\ldots\dot\cup G_k$ satisfies the conditions of Theorem \ref{theorem_corona}; then, by Proposition \ref{thm_disjoint_union} and Proposition \ref{prop_closed_form1}, $\mathcal{P}(G;x)=x^r(x^3+3x^2+3x)^k$, and therefore $G$ has exactly three distinct power domination roots: $0$, $\frac{-3+\sqrt3 i}{2}$, and $\frac{-3-\sqrt 3 i}{2}$.

Now suppose that $G$ has exactly three distinct power domination roots. If all components of $G$ have at most 2 vertices, then by Theorem \ref{thm_two_roots}, $G$ has at most two distinct power domination roots. Thus, let $G'$ be a component of $G$ with $n'\geq 3$ vertices. By Observation \ref{obs_roots3} and Theorem \ref{thm_two_roots}, $G'$ cannot have fewer than three distinct power domination roots. Thus, $G'$ must have exactly 3 distinct power domination roots, so $\mathcal{P}(G';x)=x^j(x-a)^k(x-b)^l$ for some $j,k,l\in\mathbb{N}$ with $j+k+l=n'$ and some $a,b\in\mathbb{C}\backslash \{0\}$ with $a\not=b$. Using the coefficient of $x^{n'-2}$ in $\mathcal{P}(G';x)$ and Corollary \ref{cor_degrees}, we have that
\begin{equation}
\label{eqn1}
\binom{k}{2}a^2+\binom{l}{2}b^2+klab=p(G';n'-2)=\binom{n'}{2}.
\end{equation}
Using the coefficient of $x^{n'-1}$ in $\mathcal{P}(G';x)$, Corollary \ref{cor_degrees}, and the fact that $G'$ has no isolates, we also have that
\begin{equation}
\label{eqn2}
-(ka+lb)=p(G';n'-1)=n'.
\end{equation}
Because $\mathcal{P}(G';x)$ is a monic polynomial with integer coefficients, $a$ is an algebraic integer and hence its minimal polynomial $A(x)$ over $\mathbb{Q}$ is a monic irreducible polynomial with integer coefficients. Note that by minimality, $A(x)$ cannot have $0$ as a root. Since $A(x)$ divides $\mathcal{P}(G';x)$, since irreducible polynomials with rational coefficients are separable, and since $\mathcal{P}(G';x)$ has three distinct roots, we have that either $A(x)=x-a$ or $A(x)=x^2+rx+s$ for some $r,s\in\mathbb{Z}$. Thus, we consider two cases:
\begin{enumerate}
\item Suppose that $A(x)=x^2+rx+s$. Since $a$, $b$, and $0$ are roots of $\mathcal{P}(G';x)$, and $A(x)$ divides $\mathcal{P}(G';x)$, and $0$ is not a root of $A(x)$, $b$ must be a root of $A(x)$. Thus, $A(x)$ must also be the minimal polynomial of $b$ over $\mathbb{Q}$. Thus $\mathcal{P}(G';x)=x^j(x^2+rx+s)^k$, which implies that $j+2k=n'$. By Theorem \ref{thm_zhao}, we have that $j=\gamma_P(G')\le n'/3$, which implies that $k\ge n'/3$. Note that $a+b\in\mathbb{Z}$, since $A(x)$ is monic and $r\in\mathbb{Z}$. Since $-k(a+b)=n'>0$ by (\ref{eqn2}), this implies that $-(a+b)\in\{1,2,3\}$. We consider three subcases:
\begin{enumerate}
\item If $a+b=-1$, then $k=n'$ and hence $j=-n'$, a contradiction.
\item If $a+b=-2$, then $k=n'/2$ and hence $j=0$, a contradiction. 
\item If $a+b=-3$, then $\gamma_P(G)=j=k=n'/3$. By Theorem \ref{thm_zhao}, $G'\in\mathscr{F}$, since $\mathcal{P}(K_{3,3};x)=x^6+6x^5+15x^4+20x^3+15x^2$, which has five distinct roots.
\end{enumerate}
\item Suppose that $A(x)=x-a$; then, the minimal polynomial of $b$ over $\mathbb{Q}$ must be $x-b$, and $-a,-b\in\mathbb{N}$ by Proposition \ref{thm_roots}. Since $n'\ge3$, by Theorem \ref{thm_zhao} we have $j\le n'/3$, which implies that $k+l\ge2n'/3>n'/2$. If $a\le-2$ and $b\le-2$, then $-(ka+lb)>n'$, a contradiction to (\ref{eqn2}). Therefore without loss of generality, we can assume that $a=-1$. Then by (\ref{eqn2}), $k-lb=n'$. On the other hand, by (\ref{eqn1}), $\binom{k}{2}+\binom{l}{2}b^2-klb=\binom{n'}{2}$. Thus $k^2-k+(l^2-l)b^2-2klb=n'^2-n'$. This implies that $(k-lb)^2-k-lb^2=n'^2-n'$. Substituting $n'$ for $k-lb$, we get $k+lb^2=n'$, but since $k-lb=n'$, we have that $b=0$ or $b=-1$. But then $\mathcal{P}(G';x)$ would only have two distinct roots: $-1$ and $0$; this is a contradiction.
\end{enumerate}
Thus, whenever a component $G'$ of $G$ has at least 3 vertices, $G'\in\mathscr{F}$. Suppose $G$ also has a $K_2$-component. Then $G$ would have $-2$ as a root. However, since $G$ must have a component $G'$ with at least 3 vertices, $G'$ must be in $\mathscr{F}$, and hence by Theorem \ref{theorem_corona}, $G'$ has power domination roots $0$, $\frac{-3+\sqrt3 i}{2}$, and $\frac{-3-\sqrt 3 i}{2}$; by Proposition \ref{thm_disjoint_union}, it follows that $G$ would have at least 4 distinct power domination roots, a contradiction. Thus, all components  of $G$ are either isolates or graphs in $\mathscr{F}$.
\qed

\vspace{9pt}

\noindent Note that it can be verified in polynomial time whether a graph $G$ is isomorphic to $G_1\dot\cup\ldots\dot\cup G_k\dot\cup\overline{K_r}$ for some $k\geq 1$, $r\geq 0$, and $G_i\in \mathscr{F}$ for $1\leq i\leq k$.

\vspace{9pt}

\noindent We now identify regions of the complex plane in which no power domination roots can exist.

\begin{theorem}[Rouch\'{e}'s Theorem]
\label{thm_Rouche}
Let $f$, $g$, and $h$ be analytic functions on a region $\Omega$, and let $C$ be a simple closed connected curve (i.e., a closed curve that does not intersect itself) in $\Omega$. If $f(z)=g(z)+h(z)$ in $\Omega$ and $|h(z)|>|g(z)|$ on $C$, then on the set enclosed by $C$, $f$ and $h$ have the same number of roots (counting repeated roots multiple times).
\end{theorem}
\begin{corollary}
\label{cor_Rouche}
Let $f$, $g$, and $h$ be polynomials such that $f(z)=g(z)+h(z)$, and on a circle of radius $r$ centered at $0$ in the complex plane, $|h(z)|>|g(z)|$. Then, on the disk of radius $r$ centered at $0$, $f$ and $h$ have the same number of roots. 
\end{corollary}

\noindent Note that in the case that $h(z)$ is a non-zero constant, $f(z)$ has no roots on the disk in question.

\begin{theorem}
\label{thm_roots_bound}
Let $G$ be a graph. Let $a$ be a positive real number, and let 
\begin{equation*}
f(G;a)=\frac{\sum_{i=1}^n p(G;i)a^i}{\sum_{i=1}^n\sum_{k=i}^n p(G;k)\binom{k}{i}a^{k-i}}.
\end{equation*}
If $a+bi$ is a root of $\mathcal{P}(G;x)$, then $|b|\geq \min \left\{f(G;a), \left(f(G;a)\right)^\frac{1}{n}\right\}$.

\end{theorem}

\begin{proof}
Define the polynomial 
\[P(G;y)=\sum_{j=1}^n p(G;j)(a+y)^j,\]
which is obtained by shifting $\mathcal{P}(G;x)$ to the left by $a$.
If $\mathcal{P}(G;x)$ has a complex root with real part $a>0$, $P(G;y)$ has a purely imaginary root $bi$ (which is non-zero since the power domination polynomial has no real positive roots). Then,

\begin{equation*}
P(G;y)=\sum_{j=1}^n p(G;j)\sum_{k=0}^j \binom{j}{k}a^{j-k}y^k=\sum_{k=0}^n \sum_{j=k}^n p(G;j)\binom{j}{k}a^{j-k}y^k,
\end{equation*}
where the first equality follows from the binomial expansion of $(a+y)^j$ and the second equality follows from moving $p(G;j)$ into the sum, interchanging the order of summation, and from the fact that $p(G;0)=0$. Hence, for $0\leq k\leq n$, the coefficient of $y^k$ in $P(G;y)$ is 
\[c_k=\sum_{j=k}^n p(G;j)\binom{j}{k}a^{j-k}.\]


\noindent By Corollary \ref{cor_Rouche} (with $f=P(G;y)$, $g=\sum_{k=1}^n c_ky^k$, and $h=c_0$), if $c_0 >\sum_{k=1}^n c_k y^k$
on $|y|=r$, then $P(G;y)$ has no roots inside a disk of radius less than $r$. Let $r^*=\sup\{r:c_0 >\sum_{k=1}^n c_k r^k\}$. Then, the root $bi$ of $P(G;y)$ satisfies $|b|\geq r^*$. We will now show that $r^*\geq \min\{f(G;a),(f(G;a))^\frac{1}{n}\}$.

Define $A(r):=\frac{c_0}{\sum_{k=1}^n c_k r^k}$ for $r>0$, and note that $r^*=\sup\{r:A(r)>1\}$. Note also that $A(r)$ is strictly monotonically decreasing as $r$ increases (since $r>0$). Hence, if $A(1)<1$, then $r^*<1$; if $A(1)>1$, then $r^*>1$; if $A(1)=1$ then $r^*=1$. Finally, note that $A(1)=f(G;a)$. We will now consider 3 cases. 

\noindent \emph{Case 1: $A(1)=1$}. Then $f(G;a)=(f(G;a))^\frac{1}{n}=1=r^*$.

\noindent\emph{Case 2: $A(1)<1$}. Since $r^*<1$, for any $r\leq r^*$, 
if $\frac{c_0}{\sum_{k=1}^nc_kr}>1$, then $\frac{c_0}{\sum_{k=1}^nc_kr^k}>1$. Thus, 
\begin{equation*}
f(G;a)=\frac{c_0}{\sum_{k=1}^nc_k}=\sup\{r:\frac{c_0}{\sum_{k=1}^nc_kr}>1\}\leq \sup\{r:\frac{c_0}{\sum_{k=1}^nc_kr^k}>1\}=r^*.
\end{equation*}
\noindent \emph{Case 3: $A(1)>1$}.  
If $r\geq 1$, then $\frac{c_0}{\sum_{k=1}^n c_kr^n}>1$ implies $\frac{c_0}{\sum_{k=1}^n c_kr^k}>1$. If $r<1$, since $A(1)>1$ and $A$ is strictly monotonically decreasing, $A(r)>1$. So, for $0<r<1$, $\frac{c_0}{\sum c_k r^n}>1$ implies $\frac{c_0}{\sum c_k r^k}>1$ because the latter is always true. Thus, for $r>0$,
\begin{equation*}
(f(G;a))^{\frac{1}{n}}=\left(\frac{c_0}{\sum_{k=1}^nc_k}\right)^{\frac{1}{n}}=\sup\{r:\frac{c_0}{\sum_{k=1}^nc_kr^n}>1\}\leq \sup\{r:\frac{c_0}{\sum_{k=1}^nc_kr^k}>1\}=r^*.\tag*{\qedhere}
\end{equation*}

\end{proof}


\noindent We now characterize complex power domination roots in an expression that is independent of the graph. 

\begin{corollary}
\label{cor_roots2}
Let $G$ be a connected graph. Let $a$ be a positive real number, and let 
\begin{equation*}
f(a)=\frac{\sum_{i=\lceil\frac{n}{3}\rceil}^n \binom{n-\lceil n/3\rceil}{i-\lceil n/3\rceil}a^i}{\sum_{i=1}^n\sum_{k=i}^n \binom{n}{k}\binom{k}{i}a^{k-i}}.
\end{equation*}
If $a+bi$ is a root of $\mathcal{P}(G;x)$, then
$|b|>\min\left\{f(a),(f(a))^n\right\}$.
\end{corollary}
\begin{proof}
If $G$ has fewer than 3 vertices, $\mathcal{P}(G;x)$ could only have the real roots 0 and $-2$. Thus, $n\geq 3$, so by Theorem \ref{thm_zhao}, $\gamma_P(G)\leq \frac{n}{3}$.
Moreover any superset of a power dominating set is also power dominating, and $p(G;k)\leq \binom{n}{k}$. Thus, $f(a)\leq f(G;a)$ from Theorem \ref{thm_roots_bound}, and the result follows.
\end{proof}

\section{Conclusion}
\label{section_conclusion}

In this paper, we introduced the power domination polynomial of a graph in order to study the enumeration problem associated with power domination. We explored various structural properties of $\mathcal{P}(G;x)$, related it to other graph polynomials, characterized $\mathcal{P}(G;x)$ for specific families of graphs, and analyzed some properties of power domination roots. We now offer several open questions about the power domination polynomial.

\begin{question}
\label{q1}
For any graph $G$, is $\mathcal{P}(G;x)$ unimodal?
\end{question}

\noindent There is some evidence the answer to Question \ref{q1} is affirmative, e.g. as seen in Proposition \ref{prop_unimodal}. We have also computationally verified that the power domination polynomials of all graphs on fewer than $10$ vertices are unimodal; computer code can be found at \url{https://github.com/rsp7/Power-Domination-Polynomial}.

Another direction for future work is to derive conditions which guarantee that a polynomial $P$ is or is not the power domination polynomial of some graph. For instance, Corollary \ref{cor_roots2} gave a necessary condition for this using the complex roots of the polynomial. In addition, it would be interesting to find other families of graphs which are uniquely identified by their power domination polynomials. The following question related to $\mathcal{P}$-unique graphs (and inspired by Theorem \ref{thm_isolate_unique}) could also be investigated.

\begin{question}
Given two $\mathcal{P}$-unique graphs $G_1$ and $G_2$, when is $G_1\dot\cup G_2$ $\mathcal{P}$-unique? What other graph operations preserve $\mathcal{P}$-uniqueness?
\end{question}

\noindent It would also be interesting to characterize or count all power dominating sets (or at least all minimum power dominating sets) of some other nontrivial families of graphs such as trees and grids. In Proposition \ref{prop_closed_form1}, the power domination polynomial of the corona of any graph with $K_k$, $k>1$ was computed explicitly. However, the case $k=1$ appears to be more difficult as Theorem \ref{theorem_corona} does not apply directly; this motivates the following question:

\begin{question}
For any graph $H$, is there an efficient way to compute $\mathcal{P}(H\circ K_1;x)$?
\end{question}

\noindent A graph polynomial $f(G;x)$ satisfies a \emph{linear recurrence relation} if $f(G;x)=\sum_{i=1}^k g_i(x)f(G_i;x)$, where the $G_i$'s are obtained from $G$ using vertex or edge elimination operations, and the $g_i$'s are fixed rational functions. For example, the chromatic polynomial $P(G;x)$ satisfies the deletion-contraction recurrence $P(G;x)=P(G-e;x)-P(G/e;x)$. Similarly, a \emph{splitting formula} for a graph polynomial $f(G;x)$ is an expression for $f(G;x)$ in terms of the polynomials of certain subgraphs of $G$; several such formulas were derived in Section \ref{sect_decomp}. In view of this, it would be interesting to investigate the following question:

\begin{question}
Are there linear recurrence relations for $\mathcal{P}(G;x)$, or splitting formulas for $\mathcal{P}(G;x)$ based on cut vertices or separating sets?
\end{question}
\noindent Answering these questions would be useful for computational approaches to the problem; in particular, a linear recurrence relation would allow the power domination polynomial of a graph to be computed recursively.

\bibliographystyle{abbrv}

\end{document}